\documentclass{birkjour}
\usepackage{amsmath,fancyhdr,amssymb,amsthm, geometry, enumerate, graphicx ,amsfonts,hyperref,mathrsfs, calrsfs}

\DeclareMathAlphabet{\pazocal}{OMS}{zplm}{m}{n}

\author{Radhika Vasisht}

\address{%
Department of Mathematics, \br Faculty of Mathematical Sciences, \br University of Delhi,\br Delhi 110 007, India}

\email{radhika.vasisht92@gmail.com}

\author{Ruchi Das}
\address{Department of Mathematics, \br Faculty of Mathematical Sciences, \br University of Delhi,\br Delhi 110 007, India}
\email{rdasmsu@gmail.com}

\title{Furstenberg families and transitivity in non-autonomous systems} 

\theoremstyle{definition}
\newtheorem{defn}{Definition}[section]
\providecommand{\keywords}[1]{\textbf{Keywords :} #1}
\providecommand{\msc}[1]{\textbf{MSC(2010)} #1}
\theoremstyle{plain}
\newtheorem{thm}{Theorem}[section]
\newtheorem{prop}{Proposition}[section]
\newtheorem{Cor}{Corollary}[section]
\newtheorem{exm}{Example}[section]
\newtheorem{rmk}{Remark}[section]

\begin{document}

\date{}
\maketitle

\begin{abstract}
We obtain necessary and sufficient conditions for a non-autonomous system to be $\mathcal{F}$-transitive and $\mathcal{F}$-mixing, where $\mathcal{F}$ is a Furstenberg family. We also obtain some characterizations for topologically ergodic non-autonomous systems. We provide examples/counter examples related to our results.
\end{abstract}


\msc{Primary 54H20; Secondary 37B55}

\section{Introduction}
Over recent years, the theory of non-autonomous dynamical systems has developed into a highly active field related to, yet recognizably distinct from the autonomous dynamical systems. In \cite{MR1402417}, Kolyada and Snoha gave definition of topological entropy in nonautonomous discrete systems. Since then various reasearchers have worked in this direction \cite{MR3344142, MR3174280}. Non-autonomous discrete systems appear in numerous fields of science such as biology, informatics, and quantum mechanics, especially when phase space and evolution are time dependent. We first introduce some notations. Consider the following non-autonomous discrete dynamical system (N.D.S) $(X,f_{1,\infty})$:
\begin{equation} x_{n+1}=f_n(x_n) , n \geq 1,\nonumber \end{equation}
where $(X,d)$ is a compact metric space, $f_n: X \rightarrow X$ is a continuous map and for convenience, we denote $({f_n})_{n =1}^\infty $ by $f_{1,\infty}$. Naturally, a difference equation of the form $x_{n+1}=f_n(x_n)$ can be seen as the discrete analogue of a non-autonomous differential equation $\frac{dx}{dt}=f(x,t)$. 
\\Dynamical properties of maps in dynamical systems are being widely studied by researchers. They are of great importance in the qualitative study of dynamical systems. One of the most useful and significant dynamical properties is topological transitivity. The concept of topological transitivity can be traced back to Birkhoff \cite{birkhoff1950collected, MR0209095}. It forms the basis of the study of chaos theory and decomposition theorems. Apart from standard topological transitivity, various variants of this concept have been defined and studied for autonomous dynamical systems \cite{MR2162226, MR2942654, MR3017943, MR3781316}. In recent past, variants of topological transitivity have been studied by many researchers for non-autonomous systems as well \cite{MR3269214, MR3584171,  vasisht2018}. While studying transitive dynamical systems, one may encounter a natural question: "Does transitivity happen after every regular interval?". One of the ways to answer this question and to give an appropriate definition to the term "regular interval", is to study the occurance of transitivity based on the largeness of subsets of $\mathbb{N}$. Since Furstenberg families are nothing but a collection of subsets of $\mathbb{N}$, one can classify transitive systems through Furstenberg families. The study of transitivity via Furstenberg families has helped in obtaining some very useful and applicable results for transivtive dynamical systems \cite{MR3365729, MR2831911}.  In \cite{MR3178028, MR2370619}, the authors have studied various dynamical properties of autonomous systems through Furstenberg families. Recently, some researchers have explored other dynamical properties also of non-autonomous discrete systems through Furstenberg families and have obtained some instresting results  \cite{MR3623040, MR3584037}. \\Recently, Sharma and Raghav have studied various dynamical properties of non-autonomous discrete dynamical systems in various settings \cite{sharma2017alterations, MR1, MR3661658}. Motivated by the work done in this direction, we study $\mathcal{F}$-transitive, $\mathcal{F}$-mixing and topologically ergodic non-autonomous systems.
\\In section 2, we give preliminaries required for the rest of the sections. In section 3, we study $\mathcal{F}$-transitivity for non-autonomous system generated by a uniformly convergent sequence of maps. We further give a result regarding the $\mathcal{F}$-transitivity of a non-autonomous system generated by a finite family of maps and some examples to support our result. Also, a relation between the $\mathcal{F}$-transitivity of the systems $(X,f_{1,\infty})$ and $(X,f_{k,\infty})$ is studied. In section 4, we obtain a relation between the $\mathcal{F}$-mixing of the systems $(X,f_{1,\infty})$ and $(X,f_{k,\infty})$. Further, we study  $\mathcal{F}$-mixing in a non-autonomous system generated by a finite family of maps giving some examples to support our results. In section 5, topologically ergodic non-autonomous systems are studied. We present a result for the topologically ergodic non-autonomous system generated by a finite family of maps and obtain a relation between the topologically ergodic systems $(X,f_{k,\infty})$ and $(X,f_{k,\infty})$.

\section{Preliminaries}
In this section we recall some well known notions.
\\Given a subset $A$ of a topological space $X$, int$(A)$ denote the interior of $A$ in $X$.
For any two open sets U and V of $X$, we denote, 
$N_{f_{1,\infty}}(U,V)= \{n\in \mathbb{N} : f_1^n(U)\cap V\ne \emptyset \}$.

\begin{defn}
A system $(X,f_{1,\infty})$ is said to $\textit{open}$ if for any open set $U$ in $X$, $f_{n}(U)$ is open for each $n\in\mathbb{N}$.
\end{defn}

\begin{defn}
A system $(X,f_{1,\infty})$ is said to $\textit{feeble open}$ if for any non-empty open set $U$ in $X$, $int(f_{n}(U))$ is non-empty for each $n\in\mathbb{N}$.
\end{defn}

\begin{defn}
The system $(X,f_{1,\infty})$ is said to be $\textit{topologically transitive}$ if for any two non-empty open sets $U$ and $V$ in $X$, there exists a positive integer $n\in \mathbb{N}$ such that, $f_{1}^n(U)\cap V \ne \emptyset$. Thus, the system $(X,f_{1,\infty})$ is said to be topologically transitive if for any two non-empty open sets $U_0$ and $V_0$ of $X$, $N_{f_{1,\infty}}(U_0,V_0)$ is non-empty.
\end{defn}

\begin{defn}
The system $(X,f_{1,\infty})$ is said to be $\textit{topologically mixing}$ if for any two non-empty open sets $U_0$ and $V_0$ in $X$, there exists a positive integer $N\in \mathbb{N}$ such that for any $n\geq N$, $U_n\cap V_0 \ne \emptyset$,
where $U_{i+1}=f_i(U_i)$ $ 1\leq i \leq n$ i.e. $(f_{n}of_{n-1}o\cdots of_1)(U_0)\cap V_0 \ne \emptyset$ for all  $ n\geq N $. Thus, the system $(X,f_{1,\infty})$ is said to be topologically mixing if for any two non-empty open sets $U_0$ and $V_0$ of $X$ , there is a positive integer N such that $N_{f_{1,\infty}}(U_0,V_0) \supset [N,\infty)\cap\mathbb{N}$.
\end{defn}

\begin{defn}
Let $|N_{f_{1,\infty}}(U,V)|$ be the cardinal number of the set $N_{f_{1,\infty}}(U,V)$, then
\[\limsup_{n \to \infty} \frac{|N_{f_{1,\infty}}(U,V)\cap N_n|}{n}\]  is called the upper density of $N_{f_{1,\infty}}(U,V)$, where $N_n=\{0,1,2,...,{n-1}\}$.
\end{defn}

\begin{defn}
The system $(X,f_{1,\infty})$ is said to be $\textit{topologically ergodic}$ if for any two non-empty open sets $U_0$ and $V_0$ in $X$, $N_{f_{1,\infty}}(U_0,V_0)$ has positive upper density.
\end{defn}

Let $(X,d)$ be a compact metric space and let $C(X)$ denote the collection of all continuous self-maps on $X$. For any $f,g\in C(X)$, the \textit{Supremum metric} is defined by \\$D(f,g)=\underset{x\in X}{sup}$ $d(f(x),g(x))$. It is easy to observe that a sequence $(f_n)$ in $C(X)$ converges to $f$ in $C((X),D)$ if and only if $f_n$ converges to f uniformly on $X$ and hence the topology generated by the \textit{Supremum metric} is called the topology of uniform convergence.
\\Let $\{f_n:n\in N\}$ be a family of continuous self maps on $X$. For any $k\in \mathbb{N}$, let $(f_{k,\infty})$ denote the family obtained by deleting first $k-1$ members. 

\bigskip Let $(X,d_X)$ and $(Y,d_Y)$ be compact metric spaces. For non-autonomous discrete dynamical systems $(X,f_{1,\infty})$ and $(Y,g_{1,\infty})$, put $(f_{1,\infty}\times g_{1,\infty})= (h_{1,\infty}) = (h_1,h_2,\ldots,h_n,\ldots)$, where $h_n=f_n\times g_n$ , for each $n\in \mathbb{N}$. Thus, $(X\times Y, f_{1,\infty}\times g_{1,\infty})$ is a non-autonomous dynamical system, where $(X\times Y)$ is a compact metric space endowed with the product metric $d_{X\times Y}((x,y),(x',y'))= d_X(x,x')+d_Y(y,y')$. Here, $h_{1}^{n}=h_n\circ h_{n-1}\circ \cdots \circ h_2\circ h_1 = (f_n\times g_n)\circ (f_{n-1}\times g_{n-1})\circ \cdots \circ (f_2\times g_2)\circ (f_1\times g_1)$ \cite{MR3584171}.

Now, we recall some concepts related to Furstenberg families.
Let $\mathcal{P}$ be the collection of all subsets of $\mathbb{Z^+}$. A collection $\mathcal{F}\subseteq \mathcal{P}$ is called a $\mathit{Furstenberg\hspace{0.15cm} family}$, if it is hereditary upwards, that is, $F_1 \subset F_2$ and $F_1\in \mathcal{F}$ implies $F_2\in \mathcal{F}$. A family $\mathcal{F}$ is proper if it is a proper subset of $\mathcal{P}$. Throughout this paper, all Furstenberg families are considered to be proper.
\\For a Furstenberg family $\mathcal{F}$, the dual family\begin{align*}
\mathit{k} \mathcal{F} = \{F\in \mathcal{P}:F\cap F'\neq  \emptyset, for\hspace{0.15cm} all\hspace{0.15cm} F'\in \mathcal{F}\}=\{F\in \mathcal{P}: \mathbb{Z^+}\setminus F\notin \mathcal{F}\}
\end{align*} 
Clearly, if $\mathcal{F}$ is a Furstenberg family, then so is $\mathit{k}\mathcal{F}$. One can note that $\mathit{k}(\mathit{k}\mathcal{F})=\mathcal{F}$. Note that the family $\mathcal{B}$ of all infinite subsets of $\mathbb{Z^+}$ is a Furstenberg family and $\mathit{k}\mathcal{B}$ is the family of all cofinite subsets of $\mathbb{Z^+}$. The family of all syndetic sets, the family of all thick sets are some examples of Furstenberg families.
\\For Furstenberg families $\mathcal{F}_1$ and $\mathcal{F}_2$, let $\mathcal{F}_1.\mathcal{F}_2=\{F_1 \cap F_2: F_1 \in \mathcal{F}_1, F_2 \in \mathcal{F}_2\}$. A Furstenberg family $\mathcal{F}$ is said to be filterdual if $\mathcal{F}$ is proper and $\mathit{k}\mathcal{F}.\mathit{k}\mathcal{F}\subseteq \mathit{k}\mathcal{F}$. A Furstenberg family $\mathcal{F}$ is said to be translation invariant if for any $F\in \mathcal{F}$ and any $i\in \mathbb{Z^+}, \hspace{0.15cm} F+i \in \mathcal{F}$ and $F-i\in \mathcal{F}$.
\\ The following two concepts have been defined and studied by various researchers for autonomous discrete dynamical systems. We define them for the non-autonomous discrete dynamical systems.
\begin{defn}
The system $(X,f_{1,\infty})$ is said to be $\mathcal{F}-\textit{transitive}$ if for any two non-empty open sets $U$ and $V$ in $X$, $N_{f_{1,\infty}}(U,V)$ $\in \mathcal{F}$.
\end{defn}

\begin{defn}
The system $(X,f_{1,\infty})$ is said to be $\mathcal{F}-\textit{mixing}$ if the product $(X\times X, f_{1,\infty}\times f_{1,\infty})$ is $\mathcal{F}-\textit{transitive}$.
\end{defn}

Let $X$ be a topological space and $\pazocal{K}$($X)$ denote the hyperspace of all non-empty compact subsets of $X$ endowed with the \textit{Vietoris Topology}. A basis of open sets for Vietoris topology is given by following sets:
\bigskip

\noindent $< U_1, U_2, \ldots , U_k >$ = $\{K \in$ $\pazocal{K}$($X)$: $K \subseteq \bigcup_{i=1}^{k} U_{i}$ and $K\cap U_{i}$  $\ne \emptyset$, for each $i$ $\in \{1, 2, \ldots ,k\}$\},

\bigskip\noindent where $U_1, U_2, \ldots ,U_k$ are non-empty open subsets of $X$.

Given metric space $(X, d)$, a point $x \in X$ and $A \in \pazocal{K}(X)$, let $d(x, A)$ = $\inf \{d(x, a): a \in A \}$. For every $\epsilon > 0$, let  open $d$-ball in $X$ about $A$ and radius $\epsilon$ be given by  $B_{d}(A, \epsilon) = \{x \in X: d(x, A) < \epsilon\} = \bigcup_{a \in A} B_d(a, \epsilon)$, where $B_d(a, \epsilon)$ denotes the open ball in $X$ centred at $a$ and of radius $\epsilon$. The Hausdorff metric on $\pazocal{K}$($X)$  induced by $d$, denoted by $d_H$, is defined  as follows:

\ \ \ \ \ \ \ \ \ \ \ \ \ \ \ \ \ $d_H$($A, B) = \inf \{ \epsilon > 0: A \subseteq B_{d}(B, \epsilon) \  \text{and} \  B \subseteq B_{d}(A, \epsilon)\}$,

\bigskip\noindent where $A$, $B \in$ $\pazocal{K}$($X)$. We shall recall that the topology induced by the Hausdorff metric coincides with the Vietoris topology if and only if the space $X$ is compact. Also, for a compact metric space $X$ and  $A, B \in$ $\pazocal{K}$($X)$, we get that $d_H$($A, B) < \epsilon$ if and only if $A \subseteq B_d(B, \epsilon)$ and $B \subseteq B_d(A, \epsilon)$.
\\Let $\pazocal{F}(X)$ denote the set of all finite subsets of $X$. Under Vietoris topology, $\pazocal{F}(X)$ is dense in $\pazocal{K}(X)$ \cite{MR2665229, MR1269778}.
 Given a continuous function $f: X \to X$, it induces a continuous function $\overline{f}$: $\pazocal{K}$($X) \to$ $\pazocal{K}$($X)$ defined by $\overline{f}(K) = f(K)$, for every $K \in$ $\pazocal{K}$($X)$, where $f(K)$ = $\{f(k) : k\in K\}$. Note that continuity of $f$ implies continuity of $\overline{f}$.

\bigskip Let $(X, f_{1,\infty})$ be a non-autonomous discrete dynamical system and $\overline{f}_n$ be the induced function on $\pazocal{K}$($X)$, by $f_n$ on $X$, for every $n\in\mathbb{N}$. Then the sequence $\overline{f}_{1,\infty}$ = ($\overline{f}_1, \overline{f}_2$, $\ldots ,\overline{f}_n, \ldots )$ induces a non-autonomous discrete dynamical system ($\pazocal{K}$($X), \overline{f}_{1,\infty})$, where $\overline{f}_1^n = \overline{f}_n \circ \ldots \circ \overline{f}_2\circ \overline{f}_1$.  Note that $\overline{f}_1^n = \overline{f^n_1}$.

In \cite{MR1}, authors have proved several results for topological transitivity, weak mixing, topological mixing, sensitive dependence on initial conditions and cofinite sensitivity using the following results for the system $(X,f_{1,\infty})$:
\begin{prop} \cite{MR1}
Let $(X,f_{1,\infty})$ be a N.D.S generated by a family $f_{1,\infty}$ and let $f$ be any continuous self map on $X$. If the family $f_{1,\infty}$ commutes with $f$ then for any $x\in X$ and any $k\in\mathbb{N}$, $d(f_{1}^{k}(x),f^{k}(x)) \leq \sum_{i=1}^k D(f_i,f)$.
\end{prop}

\begin{Cor} \cite{MR1}
Let $(X,f_{1,\infty})$ be a N.D.S generated by a family $f_{1,\infty}$ and let $f$ be any continuous self map on $X$. If the family $f_{1,\infty}$ commutes with $f$ then for any $x\in X$ and any $k\in\mathbb{N}$, $d(f_{1}^{n+k}(x),f^{k}(f_{1}^{n}(x))) \leq \sum_{i=1}^k D(f_{i+1},f)$.
\end{Cor}

\section{$\mathcal{F}$-transitivity for Non-autonomous Discrete Dynamical Systems}
In this section, we prove results for $\mathcal{F}$-transitivity of  non-autonomous discrete dynamical systems. The next theorem gives a necessary and sufficient condition for $\mathcal{F}$-transitivity of the non-autonomous systems $(X,f_{1,\infty})$ and its corresponding autonomous system $(X,f)$.
 
\begin{thm} Let $(X,f_{1,\infty})$ be a N.D.S generated by a family $f_{1,\infty}$ of feeble open maps commuting with $f$ such that $\sum_{i=1}^\infty D(f_i,f) < \infty $ and $\mathcal{F}$ be a filterdual and translation invariant Furstenberg family. Then $(X,f)$ is $\mathcal{F}$-transitive if and only if $(X,f_{1,\infty})$ is $\mathcal{F}$-transitive.
\end{thm}
\begin{proof} Let $(X,f)$ be $\mathcal{F}$-transitive, $\epsilon > 0$ be given and $U = B(x,\epsilon)$ and $V = B(y,\epsilon)$ be two non-empty open sets in $X$. As $\sum_{i=1}^\infty D(f_i,f) < \infty$, there exists $r$ such that $\sum_{i=r}^\infty D(f_i,f) < \epsilon/2$. As the family $f_{1,\infty}$ consists of feeble open maps, therefore $f_{1}^{r}(U)$ has non-empty interior. Let $U' = \text{int}(f_{1}^{r}(U))$ and $V' = B(y,\epsilon/2)$. Clearly, $U'$ and $V'$ are non-empty open sets in $X$ and $(X,f)$ being $\mathcal{F}$- transitive implies $N_{f}(U',V')\in \mathcal{F}$. 
Let $m \in N_{f}(U',V')$ so there exists $u' \in U'$ such that $f^m(u') \in V'$.
Now since $U'=\text{int}(f_{1}^{r}(U))$ therefore there exists $u \in U$ such that $u'= f_{1}^{r}(u)$ therefore, $f^m(f_{1}^{r}(u)) \in V'$.
Also by Corollary 2.1, we have, 
\begin{align} d(f_{1}^{m+r}(U), f^m(f_{1}^{r}(U)))&\leq  \sum_{i=1}^{m} D(f_i,f) < \epsilon/2 \nonumber \end{align} therefore by triangle inequality, we get, \begin{align} d(y,f_{1}^{m+r}(U)) &\leq d(y,f_{1}^{r}(U)) +d(f^m(f_{1}^{r}),f_{1}^{m+r}(U)) < \epsilon\nonumber\end{align} which implies $f_{1}^{m+r}(U) \cap V  \neq \emptyset$ hence $m+r \in N_{f_{1,\infty}}(U,V)$ and we get $N_f(U',V')+r \subseteq N_{f_{1,\infty}}(U,V)$ implying $N_{f_{1,\infty}}(U,V)\in \mathcal{F}$.
Thus, $(X,f_{1,\infty})$ is $\mathcal{F}$-transitive.

Conversely, let $\epsilon > 0$ be given and let $B(x,\epsilon)$ and $B(y,\epsilon)$ be two non-empty open sets in $X$. As $\sum_{i=1}^\infty D(f_i,f) < \infty $, choose $r\in \mathbb{N}$ such that $\sum_{i=r}^\infty D(f_i,f) < \epsilon/2$. Further, the $\mathcal{F}$-transitivity of the system $(X,f_{1,\infty})$ ensures that any non-empty open set $U$ visits $B(y,\epsilon)$ infinitely many times.
Applying $\mathcal{F}$-transitivity of $(X,f_{1,\infty})$ to open sets $U=(f_{1}^{r})^{-1}B(x,\epsilon)$ and $V=B(y,\epsilon/2)$, we get that $N_{f_{1,\infty}}(U,V)\in \mathcal{F}$. Choose $k$ such that $f_{1}^{r+k}(U) \cap V \neq \phi$. Consequently, there exists $u \in U$ such that $d(f_{1}^{r+k}(u),y) < \epsilon/2$. Also by Corollary 2.1, we have \begin{align}d(f_{1}^{r+k}(u),f^k(f_{1}^{r}(u))) < \sum_{i=1}^{k} D(f_{r+i},f) < \epsilon/2\nonumber\end{align} and therefore by triangle inequality \begin{align}d(y,f^k(f_{1}^{r}(u))) < \epsilon.\nonumber\end{align} As $f_{1}^{r}(u) \in B(x,\epsilon)$, we have $(f^k B(x,\epsilon)) \cap B(y,\epsilon) \neq \emptyset$ implying \begin{align}(N_{f_{1,\infty}}(U,V)-r) \subseteq N_f(B(x,\epsilon),B(y,\epsilon))\nonumber\end{align} therefore $N_f(B(x,\epsilon),B(y,\epsilon))\in \mathcal{F}$.
Hence, $(X,f)$ is $\mathcal{F}$-transitive.

\end{proof}

The following example justifies the necessity of the conditions in the hypothesis of Theorem 3.1. 

\begin{exm}\end{exm}

Let $I$ be the unit interval [0,1] and $\mathcal{F}$ be the family of all infinite subsets of $\mathbb{N}$. Let $f, g$ be defined on I by: \[ f(x) = \begin{cases}
2x, & \text{for} \ x \in \left[0, \frac{1}{4}\right] \\
1/2, & \text{for} \ x \in \left[\frac{1}{4}, \frac{1}{2}\right] \\
x, & \text{for} \ x \in \left[\frac{1}{2}, 1\right]  
\end{cases} \]

 \[ g(x) = \begin{cases}
1/2-2x,  & \text{for} \ x \in \left[0, \frac{1}{4}\right] \\
4x-1, & \text{for} \ x \in \left[\frac{1}{4},\frac{1}{2}\right] \\
-2-2x, & \text{for} \ x \in \left[\frac{1}{2},1\right]
\end{cases} \]
Take $f_1(x)=f(x)$ and $f_n(x)=g(x)$, for all $n>1$. Then ${(f_{n})}_{n=1}^{\infty}$ converges uniformly to $g$. Note that g is feeble open. Clearly, $(X,g)$ is $\mathcal{F}$-transitive. However, the system $(X,f_{1,\infty})$ is not $\mathcal{F}$-transitive because for open sets $U=(1/4,1/2)$ and $V=(1/2,3/4)$, $\mathbb{N}_{f_{1,\infty}}(U,V)$, is empty and hence it is not $\mathcal{F}$-transitive. Our Theorem 3.1 doesn't hold true here because in the family $(f_{n})$, the function $f_1$ is not feeble open.

\begin{thm}
Let $(X,f_{1,\infty})$ be a N.D.S generated by a family $f_{1,\infty}$ of feeble open maps, then $(X,f_{1,\infty})$ is $\mathcal{F}$-transitive if and only if $(X,f_{k,\infty})$ is $\mathcal{F}$-transitive.
\end{thm}

\begin{proof}
Let $(X,f_{1,\infty})$ be $\mathcal{F}$-transitive and $U,V$ be two non-empty open sets in $X$. Since $(X,f_{1,\infty})$ is $\mathcal{F}$-transitive, therefore for open sets $U^* =(f_1^{k-1})^{-1}(U)$ and $V$, $N_{f_{1,\infty}}(U^*,V)\in \mathcal{F}$.  Since $\mathcal{F}$-transitive implies transitive, so we get that the set $N_{f_{1,\infty}}(U^*,V)$ is infinite. Thus, we get an $m'\in N_{f_{1,\infty}}(U^*,V)$ such that $m'\in N_{f_{1,\infty}}(U^*,V)$ and hence $(X,f_{k,\infty})$ is $\mathcal{F}$-transitive.
\\ Conversely, suppose that $(X,f_{k,\infty})$ is $\mathcal{F}$-transitive and $U,V$ be two non-empty pen sets in $X$. We know that the family $(f_{1,\infty})$ is feeble open, therefore, $f_1^{k-1}(U)$ has non-empty interior. We denote $int(f_1^k(U))$ by $U^*$. Now, by $\mathcal{F}$-transitivity of $(X,f_{k,\infty})$, we have that $N_{f_{k,\infty}}(U^*,V)\in \mathcal{F}$. Therefore, for any $m \in N_{f_{k,\infty}}(U^*,V)$, $f_k^m(U^*)\cap V$ is non-empty and hence we have $f_k^m(f_1^{k-1}(U)\cap V)$ is non-empty and thus $m \in N_{f_{1,\infty}}(U,V)$. So, we have $N_{f_{k,\infty}}(U^*,V) \subseteq N_{f_{1,\infty}}(U,V)$ implying $(X,f_1,\infty)$ is $\mathcal{F}$-transitive. 
\end{proof}

\begin{rmk}
In the above Theorem, $\mathcal{F}$-transitivity of the system $(X,f_{1,\infty})$ implies the $\mathcal{F}$-transitivity of the system $(X,f_{k,\infty})$ even when the family $(f_{1,\infty})$ is not feeble open. However, for the converse, the condition for the family $(f_{1,\infty})$ to be feeble open is necessary. In the next example, we show that how the result fails when the family $(f_{1,\infty})$ is not feeble open.
\end{rmk}

\begin{exm} \end{exm} 
Let $I$ be the interval $[0,1]$, $\mathcal{F}$ be the family of all infinite subsets of $\mathbb{N}$ and $g_1, g_2$ on $I$ be defined by: 
\[ g_1(x) = \begin{cases}
1,  & \text{for} \ x \in \left[0, \frac{1}{3}\right] \\
-3/2x+3/2, & \text{for} \ x \in \left[\frac{1}{3},1\right].
\end{cases} \]

\[ g_2(x) = \begin{cases}
3x,  & \text{for} \ x \in \left[0, \frac{1}{3}\right] \\
-3/2x+3/2, & \text{for} \ x \in \left[\frac{1}{4},1\right].
\end{cases} \]

Take $f_1(x)=g_1(x)$ and $f_n(x)=g_2(x)$, for all $n>1$. For any $k\in \mathbb{N}$, the non-autonomous system $(X,f_{k,\infty})$ is the autonomous system generated by $g_2$ and hence it exhibits $\mathcal{F}$-transitivity. However, the system $(X,f_{1,\infty})$ is not $\mathcal{F}$-transitive because for open sets $U=(0,1/6)$ and $V=(1/2,3/4)$, $\mathbb{N}_{f_{1,\infty}}(U,V)$, is empty. Our Theorem 3.2 fails here because in the family $(f_{n})$, the function $f_1$ is not feeble open.

\begin{rmk}
The above example also works if we take $\mathcal{F}$ to be the family of all syndetic subsets  or thick subsets of $\mathbb{N}$.  
\end{rmk}

\begin{thm}
Let $(X,f_{1,\infty})$ be a dynamical system. If $(\pazocal{K}$($X),\overline{f}_{1,\infty}$) is $\mathcal{F}$-transitive, then so is $(X,f_{1,\infty})$.
\end{thm}

\begin{proof}
Let $U$ and $V$ be two non-empty open sets in $X$, then $\pazocal{U}$=$<U>$ and $\pazocal{V}$=$<V>$ are non-empty open sets in  $(\pazocal{K}$($X)$). Since $(\pazocal{K}$($X),\overline{f}_{1,\infty}$) is $\mathcal{F}$-transitive, therefore $N_{\overline{f_{1}^{\infty}}}(\pazocal{U},\pazocal{V})\in \mathcal{F}$. Let $n\in N_{\overline{f_{1}^{\infty}}}(\pazocal{U},\pazocal{V})$, so $\overline{f_{1}^{n}}(\pazocal{U},\pazocal{V})$ is non-empty. Then, there exists $K\in\pazocal{U}$ such that $\overline{f_{1}^{n}}(K)\in\pazocal{V}$ which implies there exists $x\in K\subset U$ such that $f_{1}^{n}(x)\in V$. Therefore, we have $n\in N_{f_{1,\infty}}(U,V)$ and hence $N_{\overline{f_{1}^{\infty}}}(\pazocal{U},\pazocal{V}$)$\subseteq N_{f_{1,\infty}}(U,V)$. Since $N_{\overline{f_{1}^{\infty}}}(\pazocal{U},\pazocal{V})\in \mathcal{F}$ therefore $N_{f_{1,\infty}}(U,V)\in \mathcal{F}$. Hence, $(X,f_{1,\infty})$ is $\mathcal{F}$-transitive.
\end{proof}

\begin{rmk}
Using example 3.2 from \cite{MR3584171} along with Lemma 2.6 from \cite{MR3779662} and taking $\mathcal{F}$ to be the family of all infinite subsets of $\mathbb{N}$, one can easily verify that the converse of the above theorem need not be true. 
\end{rmk}

\vspace{0.5cm}

Our next result studies the relation between the $\mathcal{F}$-transitivity of the non-autonomous system $(X,f_{1,\infty})$ generated by a finite family of maps and its corresponding non-autonomous system.

\begin{prop}
Let $(X,f_{1,\infty})$ be a non-autonomous system generated by a finite family of maps, $\mathbb{F}=\{f_1, f_2, \ldots, f_k\}$. If the autonomous system $(X,f_k\circ f_{k-1} \circ \ldots \circ f_1)$ is $\mathcal{F}$-transitive, then $(X,f_{1,\infty})$ is also $\mathcal{F}$-transitive.
\end{prop}

\begin{proof}
Let $U,V$ be any pair of non-empty open sets in $X$. Suppose $(X,f_k\circ f_{k-1} \circ \ldots \circ f_1)$ is $\mathcal{F}$-transitive, then we have $\{n\in \mathbb{N}: {(f_k\circ f_{k-1}\circ \ldots \circ f_1 )}^n\\(U)\cap V\neq \emptyset\} \in \mathcal{F}$. Therefore, we get that the set $\{m\in \mathbb{N}: f_1^m(U)\cup V \neq \emptyset\} \in \mathcal{F}$. Hence, $(X,f_{1,\infty})$ is $\mathcal{F}$-transitive.
\end{proof}

In the next example, we show that the converse of the Proposition 3.1 is not true.

\begin{exm} \end{exm} 
Let $I$ be the interval $[0,1]$, $\mathcal{F}$ be the family of all infinite subsets of $\mathbb{N}$ and $f_1, f_2$ on $I$ be defined by: 
\[ f_1(x) = \begin{cases}
2x+1/2,  & \text{for} \ x \in \left[0, \frac{1}{4}\right] \\
-2x+3/2, & \text{for} \ x \in \left[\frac{1}{4},\frac{3}{4}\right] \\
2x-3/2, & \text{for} \ x \in \left[\frac{3}{4},1\right].
\end{cases} \]

\[ f_2(x) = \begin{cases}
x+1/2,  & \text{for} \ x \in \left[0, \frac{1}{2}\right] \\
-6x+4, & \text{for} \ x \in \left[\frac{1}{2},\frac{2}{3}\right] \\
3x-2, & \text{for} \ x \in \left[\frac{2}{3},1\right].
\end{cases} \]

and \[ (f_2\circ f_1(x)) = \begin{cases}
-12x+1,  & \text{for} \ x \in \left[0, \frac{1}{12}\right] \\
6x-1/2, & \text{for} \ x \in \left[\frac{1}{12},\frac{1}{4}\right] \\
12x-5, & \text{for} \ x \in \left[\frac{1}{4},\frac{5}{12}\right] \\
-6x+5/2, & \text{for} \ x \in \left[\frac{5}{12},\frac{1}{2}\right] \\
-2x+2, & \text{for} \ x \in \left[\frac{1}{2},\frac{3}{4}\right] \\
2x-1, & \text{for} \ x \in \left[\frac{3}{4},1\right].
\end{cases} \]

Let $(X,f_{1,\infty})$ be the non-autonomous system generated by the the finite family $F=\{f_1, f_2\}$ and $(X,f_2\circ f_1)$ be the corresponding autonomous system. The system $(X,f_2\circ f_1)$ has an invariant set $[1/2,0]$ and hence it is not $\mathcal{F}$-transitive. However, as $f_1$ expands every open set $U$ in [0,1] and $f_2$ expands every open set in [1/2,1] with $f_2([0,1/2 ])$ = [ 1/2,1], the non-autonomous system $(X,f_{1,\infty})$ is $\mathcal{F}$-transitive.
\vspace{0.5cm}

In the following example, we generate a $\mathcal{F}$-transitive non-autonomous systems by two maps such that none of the two maps is $\mathcal{F}$-transitive.

\begin{exm} \end{exm} 
Let $I$ be the interval $[0,1]$, $\mathcal{F}$ be the family of all infinite subsets of $\mathbb{N}$ and $f_1, f_2$ on $I$ be defined by: 
\[ f_1(x) = \begin{cases}
4x,  & \text{for} \ x \in \left[0, \frac{1}{4}\right] \\
5/4-x, & \text{for} \ x \in \left[\frac{1}{4},1\right].
\end{cases} \]

\[ f_2(x) = \begin{cases}
1/4-x,  & \text{for} \ x \in \left[0, \frac{1}{4}\right] \\
4/3x-1/3, & \text{for} \ x \in \left[\frac{1}{4},1\right].
\end{cases} \]

Let $\mathbb{F}=\{f_1,f_2\}$ and $(X,f_{1,\infty})$ be non-autonomous system generated by $\mathbb{F}$. Note that [1/4,1] is an invariant set for $f_1$ and [0,1/4] is an invariant set for $f_2$. Therefore, neither $f_1$ nor $f_2$ is $\mathcal{F}$-transitive. However, their composition given by

\[ (f_2\circ f_1(x)) = \begin{cases}
1/4-4x,  & \text{for} \ x \in \left[0, \frac{1}{16}\right] \\
16/3x-1/3, & \text{for} \ x \in \left[\frac{1}{16},\frac{1}{4}\right] \\
4/3(1-x), & \text{for} \ x \in \left[\frac{1}{4},1\right].
\end{cases} \]

is $\mathcal{F}$-transitive and hence by Proposition 3.1, the non-autonomous system $(X,f_{1\infty})$ is $\mathcal{F}$-transitive.

\section{$\mathcal{F}$-mixing for Non-autonomous Discrete Dynamical Systems}
In this section, we obtain necessary and sufficient condition for $\mathcal{F}$-mixing of non-autonomous system $(X,f_{1,\infty})$

\begin{thm}
Let $(X,f_{1,\infty})$ be a N.D.S generated by a family $f_{1,\infty}$ of feeble open maps, then $(X,f_{1,\infty})$ is $\mathcal{F}$-mixing if and only if $(X,f_{k,\infty})$ is $\mathcal{F}$-mixing.
\end{thm}

\begin{proof}
Let $(X,f_{1,\infty})$ be $\mathcal{F}$-mixing and $U_1, U_2, V_1, V_2$ be non-empty open sets in $X$. Since $(X,f_{1,\infty})$ is $\mathcal{F}$-mixing, therefore for open sets $U^*_i =(f_1^{k-1})^{-1}(U_i)$ and $V_i$, $N_{f_{1,\infty}}(U^*_i,V_i)\in \mathcal{F}$, for all $i=1,2$.  Since $\mathcal{F}$-mixing implies transitive, so we get that the set $N_{f_{1,\infty}}(U^*_i,V_i)$ is infinite, for all $i=1,2$. Thus, for $i=1,2$ we can get $m'\in N_{f_{1,\infty}}(U^*_i,V_i)$ such that $m'\in N_{f_{1,\infty}}(U^*_i,V_i)$ and hence $(X,f_{k,\infty})$ is $\mathcal{F}$-mixing.
\\ Conversely, suppose that $(X,f_{k,\infty})$ is $\mathcal{F}$-mixing and $U_1, U_2, V_1, V_2$ be non-empty open sets in $X$. We know that the family $(f_{1,\infty})$ is feeble open, therefore, $f_1^{k-1}(U_i)$ has non-empty interior. We denote $int(f_1^k(U_i))$ by $U^*_i$. Now, by $\mathcal{F}$-mixing of $(X,f_{k,\infty})$, we have that $N_{f_{k,\infty}}(U^*_i,V_i)\in \mathcal{F}$, for $i=1,2$. Therefore, for any $m \in N_{f_{k,\infty}}(U^*_i,V_i)$, $f_k^m(U^*_i)\cap V_i$ is non-empty, for $i=1,2$, implying $f_k^m(f_1^{k-1}(U_i)\cap V_i)$ is non-empty and hence, $m \in N_{f_{1,\infty}}(U_i,_iV)$, for $i=1,2$. So, we have that $N_{f_{k,\infty}}(U^*_i,V_i) \subseteq N_{f_{1,\infty}}(U_i,V_i)$, for $i=1,2$ implying$(X,f_1,\infty)$ is $\mathcal{F}$-mixing. 
\end{proof}

\begin{rmk}
In the above result, $\mathcal{F}$-mixing of the system $(X,f_{1,\infty})$ implies the $\mathcal{F}$-mixing of the system $(X,f_{k,\infty})$ even when the condition of the family $(f_{1,\infty})$ being feeble open is dropped from the hypothesis. However, for the converse to hold true, the family $(f_{1,\infty})$ has to be feeble open.
\end{rmk}
 
In the next example, we will show how that the above result doesn't hold true when the family $(f_{1,\infty})$ is not feeble open.

\begin{exm} \end{exm} 
Let $I$ be the interval $[0,1]$, $\mathcal{F}$ be the family of all infinite subsets of $\mathbb{N}$ and $g_1, g_2$ on $I$ be defined by: 
\[ g_1(x) = \begin{cases}
1/2,  & \text{for} \ x \in \left[0, \frac{1}{4}\right] \\
2x+1/3, & \text{for} \ x \in \left[\frac{1}{4},1\right].
\end{cases} \]

\[ g_2(x) = \begin{cases}
4x,  & \text{for} \ x \in \left[0, \frac{1}{4}\right] \\
4/3(1-x), & \text{for} \ x \in \left[\frac{1}{4},1\right].
\end{cases} \]

Take $f_1(x)=g_1(x)$ and $f_n(x)=g_2(x)$, for all $n>1$. For any $k\in \mathbb{N}$, the non-autonomous system $(X,f_{k,\infty})$ is the autonomous system generated by $g_2$ and hence it exhibits $\mathcal{F}$-mixing. However, the system $(X,f_{1,\infty})$ is not $\mathcal{F}$-mixing because for open sets $U_1=(0,1/8), U_2=(1/16,1/4)$ and $V_1=(1/2,3/4), V_2=(1/4,1/2)$, $\mathbb{N}_{f_{1,\infty}}(U_i,V_i)$ is empty, for all $i=1,2$ and hence it is not $\mathcal{F}$-mixing. Our Theorem 4.1 doesn't hold here because in the family $(f_{n})$, the function $f_1$ is not feeble open.

\begin{thm}
Let $(X,f_{1,\infty})$ be a dynamical system. If $(\pazocal{K}$($X),\overline{f}_{1,\infty}$) is $\mathcal{F}$-mixing, then so is $(X,f_{1,\infty})$.
\end{thm}

\begin{proof}
Let $U$ and $V$ be two non-empty open sets in $X$, then $\pazocal{U}_i$=$<U_i>$ and $\pazocal{V}_i$=$<V_i>$ are non-empty open sets in  $(\pazocal{K}$($X)$), for $i=1,2$. Since $(\pazocal{K}$($X),\overline{f}_{1,\infty}$) is $\mathcal{F}$-mixing, therefore $N_{\overline{f_{1}^{\infty}}}(\pazocal{U}_i,\pazocal{V}_i)\in \mathcal{F}$, for all $i=1,2$. Let $n\in N_{\overline{f_{1}^{\infty}}}(\pazocal{U}_i,\pazocal{V}_i)$, so $\overline{f_{1}^{n}}(\pazocal{U}_i,\pazocal{V}_i)$ is non-empty, $i=1,2$. Then, there exists $K_i\in\pazocal{U}_i$ such that $\overline{f_{1}^{n}}(K_i)\in\pazocal{V_i}$ which implies that there exists $x_i\in K_i\subset U_i$ such that $f_{1}^{n}(x_i)\in V, i=1,2$. Therefore, we have $n\in N_{f_{1,\infty}}(U_i,V_i)$ and hence $N_{\overline{f_{1}^{\infty}}}(\pazocal{U}_i,\pazocal{V}_i$)$\subseteq N_{f_{1,\infty}}(U_i,V_i)$. Since $N_{\overline{f_{1}^{\infty}}}(\pazocal{U}_i,\pazocal{V}_i)\in \mathcal{F}$, therefore $N_{f_{1,\infty}}(U_i,V_i)\in \mathcal{F}$, for $i=1,2$, implying $(X,f_{1,\infty})$ is $\mathcal{F}$-mixing.
\end{proof}

\vspace{1cm}
Observe that for any two pair of non-empty open sets, $U_1, U_2, V_1, V_2$ in $X$, if $(X,f_k\circ f_{k-1} \circ \ldots \circ f_1)$ is $\mathcal{F}$-mixing, then the set $\{n\in \mathbb{N}:{(f_k\circ f_{k-1}\circ \ldots \circ f_1 )}^n(U_i)\cap V_i \neq \emptyset\} \in \mathcal{F}$, for all $n=1,2$. Therefore, we get that the set $\{m\in \mathbb{N}: f_1^m(U_i)\cup V_i \neq \emptyset\} \in \mathcal{F}$, for all $n=1,2$. Thus, $(X,f_{1,\infty})$ is $\mathcal{F}$-mixing and hence we get the following proposition.

\begin{prop}
Let $(X,f_{1,\infty})$ be a non-autonomous system generated by a finite family of maps, $\mathbb{F}=\{f_1, f_2, \ldots, f_k\}$. If the autonomous system $(X,f_k\circ f_{k-1} \circ \ldots \circ f_1)$ is $\mathcal{F}$-mixing, then $(X,f_{1,\infty})$ is also $\mathcal{F}$-mixing.
\end{prop}

In the following example, we generate a non-autonomous system which is $\mathcal{F}$-mixing by two maps, whereas the autonomous systems corresponding to each of the two maps are not $\mathcal{F}$-mixing.
\begin{exm} \end{exm} 
Let $I$ be the interval $[0,1]$, $\mathcal{F}$ be the family of all infinite subsets of $\mathbb{N}$ and $f_1, f_2$ on $I$ be defined by: 
\[ f_1(x) = \begin{cases}
3/2x,  & \text{for} \ x \in \left[0, \frac{2}{3}\right] \\
5/3-x, & \text{for} \ x \in \left[\frac{2}{3},1\right].
\end{cases} \]

\[ f_2(x) = \begin{cases}
2/3-x,  & \text{for} \ x \in \left[0, \frac{2}{3}\right] \\
3x-2, & \text{for} \ x \in \left[\frac{2}{3},1\right].
\end{cases} \]

Let $\mathbb{F}=\{f_1,f_2\}$ and $(X,f_{1,\infty})$ be non-autonomous system generated by $\mathbb{F}$. Note that [2/3,1] is an invariant set for $f_1$ and [0,2/3] is an invariant set for $f_2$. Therefore, neither $f_1$ nor $f_2$ is $\mathcal{F}$-mixing. However, their composition given by

\[ (f_2\circ f_1(x)) = \begin{cases}
1/4-4x,  & \text{for} \ x \in \left[0, \frac{4}{9}\right] \\
16/3x-1/3, & \text{for} \ x \in \left[\frac{4}{9},\frac{2}{3}\right] \\
4/3(1-x), & \text{for} \ x \in \left[\frac{2}{3},1\right].
\end{cases} \]

is $\mathcal{F}$-mixing and hence by Proposition 4.1, the non-autonomous system $(X,f_{1\infty})$ is $\mathcal{F}$-mixing.

\section{Some Topologically Ergodic Non-autonomous Systems}

In this section, we obtain results for topologically ergodic non-autonomous discrete dynamical systems. In the next result, we prove that under certain conditions, the non-autonomous systems $(X,f_{1,\infty})$ and $(X,f_{k,\infty})$ are simultaneously topologically ergodic.

\begin{thm}
Let $(X,f_{1,\infty})$ be a N.D.S generated by a family $f_{1,\infty}$ of feeble open maps, then $(X,f_{1,\infty})$ is topologically ergodic if and only if $(X,f_{k,\infty})$ is topologically ergodic.
\end{thm}

\begin{proof}
Let $(X,f_{1,\infty})$ be topologically ergodic and $U,V$ be two non-empty open sets in $X$. Since $(X,f_{1,\infty})$ is topologically ergodic, therefore for open sets $U^* =(f_1^{k-1})^{-1}(U)$ and $V$, $N_{f_{1,\infty}}(U^*,V)$ has positive upper density.  Since topological ergodicity implies transitivity, we get that the set $N_{f_{1,\infty}}(U^*,V)$ is infinite. Thus, we get $m'\in N_{f_{1,\infty}}(U^*,V)$ such that $m'\in N_{f_{1,\infty}}(U^*,V)$ and hence $(X,f_{k,\infty})$ is topologically ergodic.
\\ Conversely, we suppose that $(X,f_{k,\infty})$ is topologically ergodic and $U,V$ be two non-empty open sets in $X$. We know that the family $(f_{1,\infty})$ is feeble open, therefore, $f_1^{k-1}(U)$ has non-empty interior. We denote $int(f_1^k(U))$ by $U^*$. Now, since $(X,f_{k,\infty})$ is topologically ergodic, we have that $N_{f_{k,\infty}}(U^*,V)$ has positive upper density. Thus, for any $m \in N_{f_{k,\infty}}(U^*,V)$, $f_k^m(U^*)\cap V$ is non-empty. Therefore, we have $f_k^m(f_1^{k-1}(U)\cap V)$ is non-empty and hence, $m \in N_{f_{1,\infty}}(U,V)$. So, we have that $N_{f_{k,\infty}}(U^*,V) \subseteq N_{f_{1,\infty}}(U,V)$. Thus, $(X,f_{1,\infty})$ is topologically ergodic. 
\end{proof}

\begin{rmk}
In the above result, topological ergodicity of $(X,f_{1,\infty})$ implies that the system $(X,f_{k,\infty})$ is topologically ergodic even when the family $(f_{1,\infty})$ is not feeble open. However, for the converse, the condition for the family $(f_{1,\infty})$ to be feeble open is necessary as justified by the next example.
\end{rmk}

\begin{exm} \end{exm} 
Let $I$ be the interval $[0,1]$ and $g_1, g_2$ on $I$ be defined by: 
\[ g_1(x) = \begin{cases}
1/2,  & \text{for} \ x \in \left[0, \frac{1}{2}\right] \\
x, & \text{for} \ x \in \left[\frac{1}{2},1\right].
\end{cases} \]

\[ g_2(x) = \begin{cases}
2x,  & \text{for} \ x \in \left[0, \frac{1}{2}\right] \\
2(1-x), & \text{for} \ x \in \left[\frac{1}{2},1\right].
\end{cases} \]

Take $f_1(x)=g_1(x)$ and $f_n(x)=g_2(x)$, for all $n>1$. For any $k\in \mathbb{N}$, the non-autonomous system $(X,f_{k,\infty})$ is the autonomous system generated by the tent map and hence it is topologically ergodic. However, the system $(X,f_{1,\infty})$ is not $\mathcal{F}$-transitive because for open sets $U=(0,1/3)$ and $V=(2/3,3/4)$, $\mathbb{N}_{f_{1,\infty}}(U,V)$, is empty and hence it is not topologically ergodic. Our Theorem 5.1 fails here because in the family $(f_{n})$, the function $f_1$ is not feeble open.
\vspace{1cm} 

Observing that, if $U,V$ are any pair of non-empty open sets in $X$ and $(X,f_k\circ f_{k-1} \circ \ldots \circ f_1)$ is topologically ergodic, then we have that the set $\{n\in \mathbb{N}:{(f_k\circ f_{k-1}\circ \ldots \circ f_1 )}^n(U)\cap V\neq \emptyset\}$ has positive upper density. Therefore, we get that the set $\{m\in \mathbb{N}: f_1^m(U)\cup V \neq \emptyset\}$ has positive upper density. Thus, $(X,f_{1,\infty})$ is topologically ergodic and hence we get the following proposition.

\begin{prop}
Let $(X,f_{1,\infty})$ be a non-autonomous system generated by a finite family of maps, $\mathbb{F}=\{f_1, f_2, \ldots, f_k\}$. If the autonomous system $(X,f_k\circ f_{k-1} \circ \ldots \circ f_1)$ is topologically ergodic, then $(X,f_{1,\infty})$ is also topologically ergodic.
\end{prop}

In the next example, we show that the converse of the above result need not be true.

\begin{exm} \end{exm} 
Let $I$ be the interval $[0,1]$ and $f_1, f_2$ on $I$ be defined by: 
\[ f_1(x) = \begin{cases}
3x+1/2,  & \text{for} \ x \in \left[0, \frac{1}{6}\right] \\
-3/2x+5/4, & \text{for} \ x \in \left[\frac{1}{6},\frac{5}{6}\right] \\
3x-5/2, & \text{for} \ x \in \left[\frac{5}{6},1\right].
\end{cases} \]

\[ f_2(x) = \begin{cases}
x+1/2,  & \text{for} \ x \in \left[0, \frac{1}{2}\right] \\
-4x+3, & \text{for} \ x \in \left[\frac{1}{2},\frac{3}{4}\right] \\
2x-3/2, & \text{for} \ x \in \left[\frac{3}{4},1\right].
\end{cases} \]

\[ (f_2\circ f_1(x)) = \begin{cases}
-12x+1,  & \text{for} \ x \in \left[0, \frac{1}{12}\right] \\
6x-1/2, & \text{for} \ x \in \left[\frac{1}{12},\frac{1}{6}\right] \\
-3x+1, & \text{for} \ x \in \left[\frac{1}{6},\frac{1}{3}\right] \\
6x-3, & \text{for} \ x \in \left[\frac{1}{3},\frac{1}{2}\right] \\
-3/2x+7/4, & \text{for} \ x \in \left[\frac{1}{2},\frac{5}{6}\right] \\
3x-2, & \text{for} \ x \in \left[\frac{5}{6},1\right].
\end{cases} \]

Let $(X,f_{1,\infty})$ be the non-autonomous system generated by the the finite family $F=\{f_1, f_2\}$ and $(X,f_2\circ f_1)$ be the corresponding autonomous system. The system $(X,f_2\circ f_1)$ has an invariant set $[1/2,0]$ and hence it is not topologically ergodic. However, as $f_1$ expands every open set $U$ in [0,1] and $f_2$ expands the right half of the unit interval with $f_2([0,1/2 ])$ = [ 1/2,1], the non-autonomous system $(X,f_{1,\infty})$ is topologically ergodic.

In the following example we show that the non-autonomous systems $(X,f_{1,\infty})$ generated by finitely many maps can be topologically ergodic even if none of the corresponding autonomous systems is topologically ergodic.

\begin{exm} \end{exm}
Let $I$ be the interval $[0,1]$ and $f_1, f_2$ on $I$ be defined by:  
\[ f_1(x) = \begin{cases}
3x,  & \text{for} \ x \in \left[0, \frac{1}{3}\right] \\
-x+4/3, & \text{for} \ x \in \left[\frac{1}{3},1\right].
\end{cases} \]

\[ f_2(x) = \begin{cases}
-x+1/3,  & \text{for} \ x \in \left[0, \frac{1}{3}\right] \\
3/2x-1/2, & \text{for} \ x \in \left[\frac{1}{3},1\right].
\end{cases} \]

Let $\mathbb{F}=\{f_1,f_2\}$ and $(X,f_{1,\infty})$ be non-autonomous system generated by $\mathbb{F}$. Note that [1/3,1] is an invariant set for $f_1$ and [0,1/3] is an invariant set for $f_2$. Therefore, neither $f_1$ nor $f_2$ is topologically ergodic. However, their composition given by

\[ (f_2\circ f_1(x)) = \begin{cases}
-3x+1/3,  & \text{for} \ x \in \left[0, \frac{1}{9}\right] \\
9/2x-1/2, & \text{for} \ x \in \left[\frac{1}{9},\frac{1}{3}\right] \\
3/2(1-x), & \text{for} \ x \in \left[\frac{1}{3},1\right].
\end{cases} \]
\\is topologically ergodic and hence by Proposition 5.1, the non-autonomous system $(X,f_{1\infty})$ is topologically ergodic.

\section*{Acknowledgement}
The first author is funded by GOVERNMENT OF INDIA, MINISTRY OF SCIENCE and TECHNOLOGY No: DST/INSPIRE Fellowship/[IF160750].

\end{document}